\documentclass[11pt,a4paper]{article}
\usepackage{mathrsfs, amsmath, amsthm, amssymb, bm, mathtools, thm-restate}
\usepackage{ulem} 
\usepackage{cancel}

\newtheorem{theorem}{Theorem}
\newtheorem{lem}{Lemma}
\newtheorem{cor}{Corollary}
\theoremstyle{definition}
\newtheorem{definition}{Definition}

\newtheorem{clm}{Claim}

\usepackage[left=25mm, top=25mm, bottom=25mm, right=25mm]{geometry}
\usepackage[T1]{fontenc}
\usepackage[inline]{enumitem}
\usepackage[square, numbers, sort&compress]{natbib}

\usepackage[colorinlistoftodos,prependcaption,textsize=scriptsize, textwidth=20mm,]{todonotes}

\begin{document}
	\title{ Degree-truncated choosability of  graphs}
	\author{Huan Zhou$^1$ \and Jialu Zhu$^1$ \and Xuding Zhu$^1$\thanks{Grant numbers: NSFC 12371359. }}
\date{%
    $^1$School of Mathematical Sciences, Zhejiang Normal University\\[2ex]%
    \today
}
	\maketitle 
	\begin{abstract}
 A graph $G$ is called degree-truncated $k$-choosable if for every list assignment $L$ with 
  $|L(v)| \ge \min\{d_G(v), k\}$ for each vertex $v$, $G$ is $L$-colourable. Richter asked whether every 3-connected non-complete planar graph is degree-truncated 6-choosable. We answer this question in negative by constructing a 3-connected non-complete planar graph which is not degree-truncated 7-choosable. Then we prove that every 3-connected non-complete planar graph is degree-truncated 16-DP-colourable (and hence degree-truncated $16$-choosable).  We further prove that for an arbitrary proper minor closed family ${\mathcal G}$ of graphs, let $s$ be the minimum integer such that   $K_{s,t} \notin \mathcal{G}$ for some $t$, then there is a constant $k$ such that every $s$-connected graph $G \in  {\mathcal G}$ other than a GDP tree is degree-truncated DP-$k$-colourable (and hence degree-truncated $k$-choosable), where a GDP-tree is a graph whose blocks are complete graphs or cycles.
  In particular, for any surface $\Sigma$, there is a constant $k$ such that every 3-connected non-complete graph embeddable on $\Sigma$ is  degree-truncated DP-$k$-colourable (and hence degree-truncated $k$-choosable).   The $s$-connectedness for graphs in $\mathcal{G}$ (and 3-connectedness for   graphs embeddable on $\Sigma$) is   necessary, as for any positive integer $k$, $K_{s-1,k^{s-1}} \in \mathcal{G}$ ($K_{2,k^2}$ is planar)  is not degree-truncated $k$-choosable. Also, non-completeness is a necessary condition, as complete graphs are not degree-choosable. 
	\end{abstract}
	
	\section{Introduction}	
	
Assume $G$ is a graph. A {\em list assignment} of $G$ is a mapping $L$ that assigns to each vertex $v$ of $G$ a set $L(v)$ of permissible colours. An $L$-colouring of $G$ is a mapping $\phi$ that assigns to each vertex $v$ a colour $\phi(v) \in L(v)$ such that $\phi(u) \ne \phi(v)$ for every edge $uv$. Given a mapping $f: V(G) \to \mathbb{N}$, an {\em $f$-list assignment} of $G$ is a list assignment $L$ of $G$ with $|L(v)| \ge f(v)$ for every vertex $v$. We say $G$ is {\em $f$-choosable} if 
$G$ is $L$-colourable for every $f$-list assignment $L$ of $G$. In particular, we say $G$ is {\em $k$-choosable} if 
$G$ is $f$-choosable for the function $f$ defined as $f(v)=k$ for every vertex $v$, and  $G$ is {\em degree-choosable} if $G$ is $f$-choosable for the function $f$ defined as $f(v)=d_G(v)$ for every vertex $v$.
The {\em choice number} $ch(G)$ of $G$ is the minimum integer $k$ for which $G$ is $k$-choosable.
For $k \ge 3$, it is NP-hard to determine if a graph $G$ is $k$-choosable. On the other hand, degree-choosable graphs have a simple characterization. 
A connected graph $G$ is called a {\em Gallai-tree} if each block of $G$ is a complete graph or an odd cycle. The following result was proved in \cite{ERT, Vizing}.

\begin{theorem}
    \label{thm-degree}
    A connected graph $G$ is not degree-choosable if and only if $G$ is a Gallai-tree.
\end{theorem}

In this paper, we study  degree-truncated $k$-choosability of graphs, which is a combination of degree-choosability and $k$-choosability.

     \begin{definition}
         Assume $G$ is a graph, $k$ is a positive integer, and $f(v)=\min\{k, d_G(v)\}$.  An $f$-list assignment of $G$ is called a {\em degree-truncated $k$-list assignment}, and if $G$ is   $f$-choosable,  then we say that $G$ is {\em  degree-truncated $k$-choosable}. The {\em degree-truncated choice number} of $G$ is the minimum $k$ such that $G$ is degree-truncated $k$-choosable.
 \end{definition}
  
In the degree-truncated $k$-list colouring model, each vertex $v$ has a list of size $k$, however, if a vertex $v$ has degree $d_G(v) < k$, then its list size is reduced to $d_G(v)$. In the study of $k$-list colouring of graphs, we do not need to consider vertices of degree less than $k$, as they can always be coloured properly, no matter what colours are assigned to its neighbours. In the degree-truncated $k$-list colouring model, vertices of small degree are also critical to colourability.  
For example, for any positive integer $k$, $K_{2, k^2}$ is a 2-connected planar graph which is not degree-truncated $k$-choosable: Assume the two parts of $K_{2, k^2}$ are $A=\{u,v\}$ and $B= \{v_{i,j}: 1 \le i, j \le k\}$. Let $L(u)=\{a_i: 1 \le i \le k\}$, $L(v)=\{b_i: 1 \le i \le k\}$ and $L(v_{i,j}) = \{a_i, b_j\}$. Then $L$ is a degree-truncated $k$-list assignment of $K_{2, k^2}$ and $K_{2, k^2}$ is not $L$-colourable.
Also, if $G$ is a Gallai-tree, then $G$ is not degree-choosable, and hence not degree-truncated $k$-choosable for any $k$.

    The degree-truncated choice number of planar graphs   was first considered by Richter \cite{Hutchinson, SV}. 
    He asked whether every 3-connected non-complete planar graph is degree-truncated 6-choosable. Motivated by this question, Hutchinson \cite{Hutchinson} studied degree-truncated choice number of outerplanar graphs. She proved that 2-connected   maximal outerplanar graphs other than $K_3$ are degree-truncated $5$-choosable, and that 2-connected bipartite outerplanar graphs are degree-truncated $4$-choosable. 

    The degree-truncated choice number of $K_5$-minor free graphs was studied in \cite{CPTV}. Denote by $S_k$ the set of vertices of  degree less than $k$ and denote by $d(S_k)$ the smallest distance between connected components of $G[S_k]$. It was shown in \cite{CPTV} that (1) for any $k \ge 3$, there are 2-connected non-complete planar graphs $G$ with $d(S_k) =2$ and minimum degree $\delta(G) =3$ that are not degree-truncated $k$-choosable; (2) for any $k \ge 3$, there are 3-connected non-complete 
 $K_5$-minor free graphs $G$ with $d(S_k) =2$  that are not degree-truncated $k$-choosable; (3)  for $k \ge 8$, every connected $K_5$-minor free graph $G$ with $d(S_k)=3$ is degree-truncated $k$-choosable, provided that $G$ is not a Gallai-tree; (4) for $k \ge 7$, every 3-connected non-complete $K_5$-minor free graph $G$ with $d(S_k) \ge 3$ is degree-truncated $k$-choosable.  

  Richter's question remained open, and it was unknown whether there is a constant $k$ such that every 3-connected non-complete planar graph is degree-truncated $k$-choosable. 
  
In this paper, we construct a 3-connected non-complete planar graph which is not degree-truncated $7$-choosable. This answers Richter's question in negative (even if $6$ is replaced by $7$). On the other hand, we prove the following result.

\begin{theorem}
    \label{thm-planar}
    Every 3-connected non-complete planar graph is degree-truncated $16$-choosable.
\end{theorem} 

A graph $H$ is a {\em minor} of a graph $G$ if by deleting some vertices and edges and contracting some other edges of $G$, we can obtain a graph isomorphic to $H$.
Assume ${\mathcal G}$ is a family of graphs. If $G \in \mathcal{G}$  implies $G' \in \mathcal{G}$ for every minor $G'$ of $G$, then we say that $\mathcal{G}$ is {\em minor closed}. A minor closed family $\mathcal{G}$ of graphs is {\em proper} if some graphs are not contained in $\mathcal{G}$.  We shall prove the following result, which extends Theorem \ref{thm-planar} to proper minor closed families of graphs.

\begin{theorem}
    \label{thm-minor}
    For any proper minor closed family $\mathcal{G}$ of graphs, there is a constant $k$  such that every $s$-connected   graph in $\mathcal{G}$ other than a Gallai-tree is degree-truncated $k$-choosable, where $s$ is the smallest integer such that $K_{s,t} \notin \mathcal{G}$ for some positive integer $t$. In particular, for any fixed surface $\Sigma$, there is a constant $k$ such that every 3-connected non-complete graph embedded in $\Sigma$ is degree-truncated $k$-choosable. 
\end{theorem}
 
The connectivity requirements in Theorem \ref{thm-planar} and Theorem \ref{thm-minor}  are necessary: for every integer $k$, $K_{s-1, k^{s-1}}$ is an $(s-1)$-connected graph $G \in \mathcal{G}$ ($K_{2, k^2}$ is a 2-connected planar graph) which is not degree-truncated $k$-choosable. 
 Indeed, if   $G = K_{s-1, k^{s-1}}$ is the complete bipartite graph with partite sets $A=\{v_1, v_2,\ldots, v_{s-1}\}$ and $B=\{u_x: x=(x_1,x_2,\ldots, x_{s-1}) \in [k]^{s-1}\}$, then let $L$ be the list assignment of $G$ with $L(v_i) = [k] \times \{i\}$, and let $L(u_x) = \{(x_i,i): i=1,2,\ldots, s-1\}$. Then $L$ is a degree-truncated $k$-list assignment of $G$, and $G$ is not $L$-colourable.  
As mentioned above, the condition of not being a Gallai tree is also necessary. In case $s \ge 3$, then $s$-connected Gallai-trees are complete graphs. So we replace the condition of not being a Gallai tree by not being a complete graph. 

Instead of directly proving Theorems \ref{thm-planar} and \ref{thm-minor}, we shall prove the corresponding results for DP-colourability.
DP-colouring of graphs is a variation of list colouring of graphs introduced by Dvo\v{r}\'{a}k and Postle \cite{DP}. 

\begin{definition}
	\label{def-cover}
	A {\em  cover} of a graph $G $ is a pair $(L,M)$, where $L = \{L(v): v \in V(G)\}$ is a family of pairwise disjoint sets ($L(v)$ can be viewed as a set of permissible colours for $v$, but distinct vertices have disjoint colour sets), and $M=\{M_{e}: e \in E(G)\}$, where for each edge $e=uv$, $M_e$ is a   matching between $L(u)$ and $L(v)$. For  $f \in \mathbb{N}^G$, we say that $(L,M)$ is an $f$-cover of $G$ if $|L(v)|\ge f(v)$ for each vertex $v \in V(G)$.
\end{definition}

For a cover $(L,M)$ of a graph $G$, we treat $(L,M)$ itself as a graph, with vertex set $\cup_{v \in V(G)}L(v)$ and edge set
$\cup_{e \in E(G)}E(M_e)$. We shall use ordinary graph operations on $(L,M)$, such as deleting a vertex, etc. In particular, for a colour $c \in L(v)$,    $N_{(L,M)}(c)$ denotes the set of those colours $c'$ such that $cc' \in M_e$ for some edge $e=vu$ of $G$, and for a subset $C $ of $ \cup_{v \in V(G)}L(v)$, $N_{(L,M)}(C) = \bigcup_{c \in C}N_{(L,M)}(c)$.

\begin{definition}
	\label{def-colouring}
	Given a cover $(L,M)$ of a graph $G$, an $(L,M)$-colouring of $G$ is a mapping $\phi: V(G) \to \bigcup_{v \in V(G)}L(v)$ such that for each vertex $v \in V(G)$, $\phi(v) \in L(v)$, and for each edge $e=uv \in E(G)$, $\phi(u)\phi(v) \notin  M_e$. We say $G$ is {\em $(L, M)$-colourable} if it has an $(L,M)$-colouring.
\end{definition}

 For an $(L,M)$-colouring $\phi$ of $G$, we say that vertex $v$ of $G$ is coloured by colour $\phi(v)$.

\begin{definition}
	\label{DP-colouring}
	Assume $G$ is a graph and $f \in  \mathbb{N}^G$.  We say $G$ is {\em DP-$f$-colourable} if for every   $f$-cover $(L,M)$, $G$ has an $(L,M)$-colouring. The {\em  DP-chromatic number } of $G$ is defined as 
 $$\chi_{DP}(G)=\min\{k: G \text{ is DP-$k$-colourable} \}.$$
 For a positive integer $k$, we say $G$ is degree-truncated DP-$k$-colourable if $G$ is DP-$f$-colourable, where $f(v)= \min\{k, d_G(v)\}$ for each vertex $v$.
\end{definition}

Given an $f$-list assignment $L$ of a graph $G$, let $(L',M)$ be the  $f$-cover of $G$ induced by $L$, where    
$L'=\{L'(v): v \in V(G)\}$ is defined as $L'(v) = \{i_v: i \in L(v)\}$
 for each vertex $v$ of $G$, and $M=\{M_{uv}: uv \in E(G)\}$ is defined as $$M_{uv}=\{\{i_u,i_v\}: i \in L(u) \cap L(v)\}$$
for each edge $\{u,v\}$ of $G$.
It is obvious that $G$ is $L$-colourable if and only if $G$ is $(L',M)$-colourable. Therefore, if $G$ is DP-$f$-colourable, then it is $f$-choosable, and hence $ch(G) \le \chi_{DP}(G)$.
On the other hand, it is known \cite{Bernshteyn2016} that the difference $\chi_{DP}(G)-ch(G)$ can be arbitrarily large. 

DP-colouring was introduced in \cite{DP} as a tool for studying list colouring of graphs. DP-colouring transforms the information from lists of vertices of a graph $G$ to matchings of edges of $G$. The names of colours in the lists become irrelevant, and we can perform the operation of identifying non-adjacent vertices in inductive proofs. In the study of degree-truncated choice number of outerplanar graphs, this feature turns out to be very useful. The result of Hutchinson was generalized in \cite{LWZZ}, where it was proved that every 2-connected $K_{2,4}$-minor free graph other than a cycle is degree-truncated DP-$5$-colourable.
Now the concept of DP-chromatic number of graphs is a graph invariant of independent interests, and has been studied a lot in the literature.

In this paper, we consider the degree-truncated DP-colourability of planar graphs and proper minor closed families of graphs. We shall prove the following results.

\begin{theorem}
    \label{thm-planarDP}
    Every 3-connected non-complete planar graph is degree-truncated DP-$16$-colourable.
\end{theorem}

\begin{theorem}
    \label{thm-minorDP}
    For any proper minor closed family $\mathcal{G}$ of graphs, there is a constant $k$  such that every $s$-connected   graph in $\mathcal{G}$ other than a GDP-tree is degree-truncated DP-$k$-colourable, where $s$ is the smallest integer such that $K_{s,t} \notin \mathcal{G}$ for some positive integer $t$. In particular, for any fixed surface $\Sigma$, there is a constant $k$ such that every 3-connected non-complete graph embedded in $\Sigma$ is degree-truncated DP-$k$-colourable. 
\end{theorem}

It is obvious that Theorem \ref{thm-planar} and Theorem \ref{thm-minor} are corollaries of
Theorem \ref{thm-planarDP} and Theorem \ref{thm-minorDP}, respectively.

\section{A 3-connected planar graph which is not degree-truncated 7-choosable} 

Richter asked whether every 3-connected non-complete planar graph is degree-truncated 6-choosable. 
Note that if a graph $G$ is degree-truncated $k$-choosable, 
then for any integer $k' >k$, $G$ is  degree-truncated $k'$-choosable. 
This section presents a 3-connected non-complete planar graph that is not degree-truncated 7-choosable. So the answer to Richter's question is negative even if 6 is replaced by 7.

 \begin{theorem}
     \label{thm-example}
     There is a  3-connected non-complete planar graph which is not degree-truncated $7$-choosable.
 \end{theorem}

 \begin{figure}
     \centering
     \includegraphics[width=0.7\linewidth]{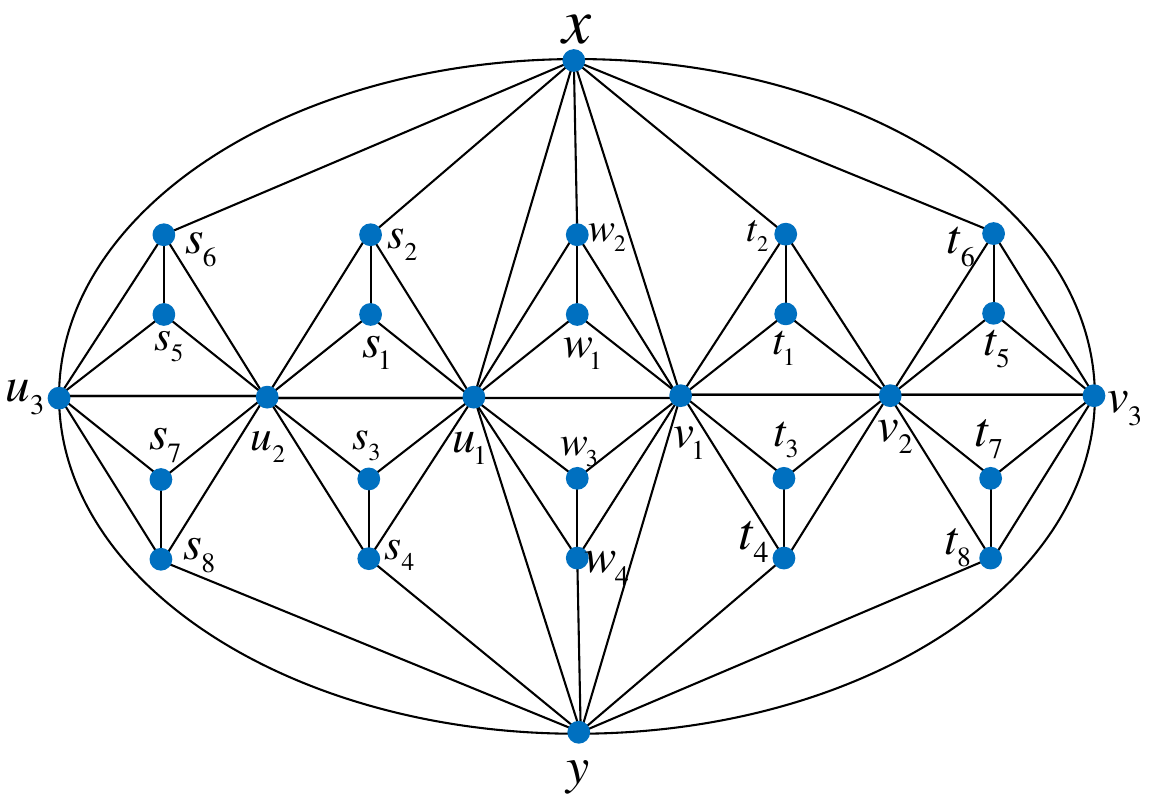}
     \caption{The graph $H$}
     \label{H}
 \end{figure}
 
 \begin{proof}
     Let $H$ be the graph in Figure \ref{H}. For $a \ne b$,  let $L$ be a list-assignment of $H$ defined as follows:
     \begin{itemize}
         \item $L(x)=\{a\}, \ L(y)=\{b\}$,
         \item $L(u_i)=L(v_i)=\{a,b,1,2,3,4,5\}$ for $i=1,2,3$,
         \item $L(w_1)=L(s_1)=L(t_1)=\{1,2,3\}, \ L(w_2)=L(s_2)=L(t_2)=\{a,1,2,3\}, \ L(w_3)=\{3,4,5\}, \ L(w_4)=\{b,3,4,5\}$,
         \item $L(s_3)=L(t_3)=\{1,2,4\}, L(s_4)=L(t_4)=\{b,1,2,4\}$,
         \item $L(s_5)=L(t_5)=\{1,2,5\}, \ L(s_6)=L(t_6)=\{a,1,2,5\}, \ L(s_7)=L(t_7)=\{3,4,5\}, \ L(s_8)=L(t_8)=\{b,3,4,5\}$.
     \end{itemize}

     First we show that $H$ is not $L$-colourable. Assume to the contrary that there is an $L$-colouring $\phi$ of $H$. 

     \begin{clm}
     \label{clm-1}
     $\phi(u_1)\in \{1,2\}$ or $\phi(v_1)\in \{1,2\}$.
     \end{clm}
     \begin{proof}
         Let $H_1=H[\{u_1,v_1,w_1,w_2,w_3,w_4\}]$. Then $\phi$ is an $L_1$-colouring of $H_1$ where $L_1(u_1)=L_1(v_1)=\{1,2,3,4,5\}, L_1(w_1)=L_1(w_2)=\{1,2,3\}$ and $L_1(w_3)=L_1(w_4)=\{3,4,5\}$.

Assume  $\phi(u_1) \not\in \{1,2\}$ and $\phi(v_1)\not\in \{1,2\}$. Then $\{\phi(u_1), \phi(v_1), \phi(w_3), \phi(w_4)\} \subseteq \{3,4,5\}$. But $\{u_1,v_1,w_3,w_4\}$ induces a copy of $K_4$, a contradiction. 
     \end{proof}

By symmetry, we assume  that $\phi(u_1)\in \{1,2\}$. Let $H_2=H[\{u_1,u_2,s_1,s_2,s_3,s_4\}]$. Then $\phi$ is an $L_2$-colouring of $H_2$ where  $L_2(u_1)=\{1,2\}, L_2(u_2)=\{1,2,3,4,5\}, L_2(s_1)=L_2(s_2)=\{1,2,3\}$ and $L_2(s_3)=L_2(s_4)=\{1,2,4\}$.
As $\{s_1,s_2,u_1\}$ induces a triangle, $\{\phi(s_1),\phi(s_2),\phi(u_1)\}=\{1,2,3\}$. Similarly, $
\{\phi(s_3), \phi(s_4), \phi(u_1)\}=\{1,2,4\}$. Therefore  $\phi(u_2)=5.$

Now, we consider graph $H_3=H[\{u_2,u_3,s_5,s_6,s_7,s_8\}]$. We have $\{\phi(s_5), \phi(s_6)\} = \{1,2\}$ and $\{\phi(s_7), \phi(s_8)\}=\{3,4\}$.
Then all the colours in the list of $u_3$ are used by its neighbours and $u_3$ cannot be properly coloured, a contradiction. 
Hence $H$ is not $L$-colourable.

     Let $G$ be a graph obtained from the disjoint union of $42$ copies $H_i$ of $H$ by identifying all the copies of $x$ into a single vertex (also named as $x$) and all the copies of $y$ into a single vertex (also named as $y$), and   adding edges $v^{(i)}_3u^{(i+1)}_3$  (where $u^{(i)}_3$ and $v^{(i)}_3$ are the copies of $u_3$ and $v_3$ in $H_i$)  for $i=1,2,\ldots, 41$, and adding an edge connecting $x$ and $y$. Then $G$ is a non-complete planar graph. It is obvious that the graph $H$ is 3-connected. For any 2-subset $S$ of $V(G)$, if $S=\{x,y\}$, then obviously $G-S$ is connected. If $S\ne \{x,y\}$, say $x \notin S$, then distinct copies of $H$ are connected via $x$ in $G-S$. As for each copy $H_i$ of $H$, $H_i-S$ is connected, we conclude that $G-S$ is connected.
     Thus $G$ is 3-connected.

     Let $L(x)=L(y)=\{a,b,c,d,e,f,g\}$. There are $42$ possible $L$-colourings $\phi$ of $x$ and $y$. Each such a colouring $\phi$ corresponds to one copy of $H$. We define the list assignment of the corresponding  copy of $H$ as $L$ by replacing $a$ with $\phi(x)$ and replacing $b$ with $\phi(y)$ (since $xy$ is an edge of $G$, $\phi(x) \ne \phi(y)$).
     It is easy to verify that $|L(v)|= \min\{d(v),7\}$ for any $v\in V(G)$. As every possible $L$-colouring of $x$ and $y$ cannot be extended to an $L$-colouring of some copy of $H$, we conclude that $G$ is not $L$-colourable. Hence $G$ is not degree-truncated $7$-choosable.
     \end{proof}

 \section{Proof of Theorem \ref{thm-planarDP}}

This section proves  3-connected non-complete planar graphs are degree-truncated DP-$16$-colourable. 

A graph $G$ is called  a {\em GDP-tree} if each block of $G$ is either a complete graph or a cycle. 
The following lemma was proved in \cite{BKP}, which is the DP-colouring version of Theorem \ref{thm-degree}. 

\begin{lem} \label{lem-dpdegree}
    Assume $G$ is a connected graph and $f \in \mathbb{N}^G$ satisfies $f(v) \ge d_G(v)$ for all $v$. Then $G$ is  DP-$f$-colourable, unless $f(v)=d_G(v)$ for each vertex $v$ and $G$ is a GDP-tree. 
\end{lem}
 
Assume $G$ is a non-complete 3-connected planar graph  and $(L,M)$ is a cover of $G$ with $|L(v)| =  \min \{16, d_G(v)\}$ for each vertex $v$.   
We shall show that $G$ is $(L,M)$-colourable.

Let $V_1=\{v \in V(G): d_G(v) < 16\}$, $  V_2 = \{v \in V(G): d_G(v) \ge 16\}$. Assume $G$ is embedded in the plane. 
We say that two vertices $u,v$ are {\em visible} to each other if they are incident to the same face of $G$. We say that $u$ is visible to a subset $X$ of $G$ if $u$ is visible to some vertex in $X$. By adding edges between vertices of $V_2$ (if needed), we may assume that 
\begin{enumerate}
    \item[(A1)] any two vertices of $V_2$ visible to each other are adjacent,
     and 
     \item[(A2)] each face of $G[V_2]$ contains at most one connected component of $G[V_1]$.
\end{enumerate} 

For (A2) to hold without adding parallel edge, the condition that $G$ be 3-connected is needed.

A \emph{partial $(L,M)$-colouring} $\phi$ of $G$ is an $(L,M)$-colouring of $G[X]$ for a subset $X$ of $V(G)$. Vertices in $X$ are called coloured vertices and other vertices are uncoloured. Given a partial $(L,M)$-colouring  $\phi$   of $G[X]$, we denote by $(L^{\phi}, M^{\phi})$ the cover
of $G-X$ defined as follows:
$$\forall u \in V(G-X), L^{\phi}(u)=L(u)-N_{(L,M)}(\phi(X)),$$
and for each edge $e=xy$ of $G-X$, $M^{\phi}_e$ is the restriction of $M_e$ to $L^{\phi}(x) \cup L^{\phi}(y)$. 

  Let $\prec$ be an order on  $V_2$ such that each vertex $v$ has at most 5 neighbours $w$ with $w \prec v$.  Since every planar graph is 5-degenerate, such an order $\prec$ exists. Moreover, we assume that vertices of a connected component of $G[V_2]$ are consecutive in this ordering.
 
  We shall colour the vertices of $G$ one by one. For each vertex $v$ of $G$, we denote by $X_v$ the set of vertices coloured before $v$, and denote by $\phi_v$ the partial $(L,M)$-colouring of $G[X_v]$.
Let $Y_v = X_v \cup \{v\}$, and denote by $\psi_v$ the partial $(L,M)$-colouring of $G[Y_v]$.

A connected component $Q$ of $G[V_1]$ is called {\em safe} with respect to a partial $(L,M)$-colouring $\phi$ of $G[X]$ if either $Q-X$ is not a GDP-tree, or there is a vertex $v \in Q-X$ with $|L^{\phi}(v)| > d_{G-X}(v)$. 

Assume we have constructed a partial $(L,M)$-colouring $\phi$ that colours $G[X]$. We choose the next vertex to be coloured as follows:
\begin{itemize}
\item[(R1)] If there is a connected component $Q$ of $G[V_1]$ and a vertex $v \in Q-X$ for which the following hold:
\begin{itemize}
    \item $Q$ is not safe with respect to $\phi$,
    \item $v$ is not a cut-vertex of $Q-X$, 
    \item $v$ is not adjacent to any vertex in $V_2-X$,
\end{itemize}   
then colour $v$ by an arbitrary colour in $L^{\phi}(v)$.
\item[(R2)] If (R1) does not apply, then we colour the least uncoloured vertex $v$ of $V_2$.
\end{itemize}

When (R2) is applied, the colour for $v$ will be chosen carefully. We shall describe the rule for choosing the colour for $v$ later.

By applying the above rules for choosing the next vertex to be coloured, the following always hold:
\begin{itemize}
\item[(C1)] For each connected component $Q$ of $G[V_1]$, for any vertex $v$ of $G$, $Q-X_v$ is connected.
\item[(C2)] For each vertex $v$ of $V_1$, for any vertex $u$ coloured before $v$,  $|L^{\psi_u}(v)| \ge d_{G-Y_u}(v)$. 
\item[(C3)] If $Q$ is a safe component of $G[V_1]$ with respect to $\phi_v$, then $Q$ remains safe with respect to $\psi_v$.
\item[(C4)] For each vertex $v$ of $V_2$, $|L^{\phi_v}(v)| \ge 11$.
\end{itemize}
(C1) follows trivially from (R1).   (C2) and (C3) hold because for any vertex   $u$ coloured before $v$,  if $u$ is not adjacent to $v$, then $L^{\psi_u}(v) = L^{\phi_u}(v)$ and $ d_{G-Y_u}(v) = d_{G-X_u}(v)$; if $u$ is  adjacent to $v$, then $|L^{\psi_u}(v)| \ge  |L^{\phi_u}(v)|-1$ and $ d_{G-Y_u}(v)=d_{G-X_u}(v) -1$.   (C4) holds because no neighbours of $v$ in $V_1$ are coloured before $v$, and at most 5 neighbours of $v$ in $V_2$ are coloured before $v$. Hence $|L^{\phi_v}(v)| \ge 16-5=11$.

Our goal is to colour all the vertices of $V_2$ and some vertices of $V_1$ so that  every connected component $Q$ of $G[V_1]$ becomes safe. If 
  this goal is achieved and $v$ is the last vertex of $V_2$, then by Lemma \ref{lem-dpdegree},  for each connected component $Q$ of $G[V_1]$, $Q-Y_v$ is $(L^{\psi_v}, M^{\psi_v})$-colourable, and hence $G-Y_v$ is $(L^{\psi_v}, M^{\psi_v})$-colourable. This implies that $G$ is $(L,M)$-colourable.

\begin{definition}
    \label{def-protector}
    Assume $Q$ is a connected component of $G[V_1]$ and $v \in V_2$. If $Q$ is non-safe with respect to $\phi_v$, and safe with respect to $\psi_v$, then we say $Q$ is {\em protected} by $v$, and $v$ is the {\em protector} of $Q$. 
\end{definition}

For $v \in V_2$ to protect $Q$, it entails finding a vertex $u \in V(Q)-X_v$ adjacent to $v$, and colour $v$ by a colour   $c \in L^{\phi_v}(v)-N_{(L,M)}(L^{\phi_v}(u))$. By colouring $v$ by such a colour $c$, we have $|L^{\psi_v}(u)| >   d_{G-Y_v}(u)$, and therefore $Q$ is safe with respect to $\psi_v$. Therefore, for $v$ to protect $Q$, in the worst case, $|L^{\phi_v}(u)| = d_{G-X_v}(u) $ colours in $L^{\phi_v}(v)$ cannot be used. We say $d_{G-X_v}(u)$ is the {\em cost} for $v$ to protect $Q$.

 Choosing carefully $v$ and $u$, we shall show that the cost of $v$ to protect a component $Q$ of $G[V_1]$ is at most 5. As   $|L^{\phi_v}(v)| \ge 11$,  we can choose a colour for $v$ from $L^{\phi_v}(v)$ to protect two connected components of $G[V_1]$. 

To complete the proof of Theorem \ref{thm-planarDP}, it 
  remains to assign to each connected component $Q$ of $G[V_1]$ a protector $v \in V_2$ so that 
  \begin{enumerate}
      \item[(D1)] each vertex $v$ is the protector of at most two connected components of $G[V_1]$.
      \item[(D2)] the cost for $v$ to protect a component of $G[V_1]$ is at most 5.
  \end{enumerate}

We need  a lemma for the  construction of such an assignment. 
 
\begin{definition}
    \label{def-thetag}
    Assume $\Gamma$ is a plane graph. We denote by $F(\Gamma)$ the set of faces of $\Gamma$. For $\theta \in F(\Gamma)$, let $V(\theta)$ be the vertices on the boundary of $\theta$. 
Let $\Theta(\Gamma)$ be the bipartite graph with partite sets $V(\Gamma)$ and $F(\Gamma)$, with $v\theta \in E(\Theta(\Gamma))$ if $v \in V(\theta)$. 
\end{definition}

    \begin{definition}
    \label{def-nice}
    A subgraph $H$ of $\Theta(\Gamma)$ is {\em nice} if $d_{H}(v) \le 2$ for each vertex $v$ of $\Gamma$, and $d_{H}(\theta) \ge  d_{\Theta(\Gamma)}(\theta)-2$ for each face $\theta $ of $\Gamma$. Moreover, $N_{\Theta(\Gamma)}(\theta) - N_H(\theta) \subseteq V(B)$ for a block $B$ of $\Gamma$.
\end{definition}

\begin{lem}
    \label{lem-nice}
    For any plane graph $\Gamma$, $\Theta(\Gamma)$ has a nice subgraph.
\end{lem}
 
The proof of this lemma is delayed to the next section. We continue the proof of Theorem \ref{thm-planarDP} by using this lemma. 

Let $H$ be a nice subgraph of $\Theta(G[V_2])$.
For a connected component $Q$ of $G[V_1]$, let $\theta_Q$ be the face of $G[V_2]$ containing $Q$. 

\begin{definition}
    \label{def-potential}
    For a connected component $Q$ of $G[V_1]$ which is a GDP-tree, vertices in $N_H(\theta_Q)$ are called {\em potential protectors} of $Q$.
\end{definition}

By the definition of $H$, each vertex $v \in V_2$ is the potential protector of at most 2 connected components of $G[V_1]$. 

\medskip
\noindent
{\bf Protector assignment rule:}
    {\em Assume $Q$ is a connected component of $G[V_1]$, which is a GDP-tree. If $v \in N_H(\theta_Q)$, $Q$ is non-safe with respect to $\phi_v$, and $v$ has a neighbour $u \in V(Q)-X_v$ with $d_{G-X_v}(u) \le 5$, then $v$ is assigned to be the protector of $Q$, and is coloured by a colour $c \in L^{\phi_v}(v) - N_{(L,M)}(L^{\phi_v}(u))$.}

\medskip

Since the protector of $Q$ 
 is chosen from $N_H(\theta_Q)$,   each vertex $v \in V_2$ is the protector of at most 2 connected components of $G[V_1]$. 
 Moreover, by the rule of choosing colour for $v$, $Q$ is indeed protected by $v$, and the cost of $v$ be a protector of $Q$ is at most $5$.   It remains to show that if a connected component $Q$ of $G[V_1]$ is a GDP-tree, then $Q$ will eventually be assigned a protector.

Assume to the contrary that $Q$ is a connected component of $G[V_1]$ which is a GDP-tree, and $Q$ is not assigned a protector by applying the protector assignment rule. In other words, for any vertex $v \in N_H(\theta_Q)$, if $u \in V(Q)-X_v$ is adjacent to $v$, then $d_{G-X_v}(u) \ge 6$.

As $|V(\theta_Q) - N_H(\theta_Q)| \le 2$, we may assume that $v_1,v_2 \in V(\theta_Q)$ and $V(\theta_Q) - N_H(\theta_Q) \subseteq \{v_1, v_2\}$.

Assume $G_1, G_2, \ldots, G_k$ are connected components of $G[V_2]$ that contain vertices adjacent to $Q$.  
We may assume that $x \prec y$ if $x \in V(G_i), y \in V(G_j)$ and $i < j$. Hence $G_k$ is the last connected component of $G[V_2]$ that contains vertices adjacent to $Q$. Let $Q'$ be the subgraph of $Q$ induced by vertices visible to $V(G_k)$.  

\begin{definition}
    \label{def-leaf}
    Assume $B$ is a leaf block of $Q'$. If  $B$ contains a cut-vertex $u$ of $Q'$, then we call $u$ the {\em root vertex} of $B$, and other vertices of $B$ are called {\em non-root vertices} of $B$.  If $B$ contains no cut-vertex of $Q'$, then $Q'=B$ and all vertices of $B$ are non-root vertices of $B$. We denote by $U(B)$ the set of non-root vertices of $B$.
\end{definition}

Let $\theta_{Q'}$ be the face of $G_k$ that contains $Q'$. Note that $\theta_{Q'} \subseteq \theta_{Q}$ and $V(\theta_{Q'}) = V(\theta_Q) \cap V(G_k)$, and that every vertex in $V(\theta_{Q'})$ is adjacent to some vertex of $Q'$. Let $N_H(\theta_{Q'}) = N_H(\theta_{Q}) \cap V(\theta_{Q'})$. 

\begin{lem}
    \label{lem-three}
    There is a vertex $v \in N_H(\theta_{Q'})$ adjacent to a non-root vertex of a leaf block of $Q'$.  
\end{lem}
\begin{proof}  
    Assume to the contrary that no vertex in $N_H(\theta_{Q'}) $ is adjacent to a non-root vertex of a leaf block of $Q'$. Then $Q'$ has at least two leaf blocks, for otherwise, every vertex of $Q'$ is a non-root vertex and  every vertex in $N_H(\theta_{Q'})$ is adjacent to a non-root vertex of $Q'$. Let   $B_1,B_2$ be two leaf blocks of $Q'$. Then $N_G(U(B_i)) \subseteq \{v_1, v_2\}$ for $i=1,2$. 
    
    If $N_G(U(B_i)) = \{v_1, v_2\}$ for $i=1,2$, then
there is a cycle $C$ induced by a subset of $\{v_1,v_2\} \cup U(B_1) \cup U(B_2)$  such that all vertices of $Q'-C$ are contained in the interior of $C$, and the vertices of $V(\theta_{Q'})-\{v_1,v_2\}$ are contained in the exterior of $C$, and not adjacent to $U(B_1) \cup U(B_2)$. This contradicts the fact that every vertex of $V(\theta_{Q'})$ is adjacent to some vertex of $Q'$.   

    Without loss of generality, assume $v_2 \notin N_G(U(B_1))$. 
  Let $u$ be the root vertex of $B_1$. Since $G$ is 3-connected, there is a path $P$ in $G-\{u,v_1\}$ that connects a vertex of $v \in V(\theta_{Q'})$ and $U(B_1)$. We may assume that $P-\{v\}$ contains no vertex of $V(\theta_{Q'})$, and all vertices of $P-\{v\}$ are visible to some vertex in $V(\theta_{Q'})$.
  Then, all the vertices in $P-\{v\}$ are vertices of $Q'$. Since $u$ separates $U(B_1)$ from $Q'-B_1$, we conclude that $P$ is an edge that connects $v \in N_{H}(\theta_{Q'})$ and $u\in U(B_1)$, a contradiction.  
\end{proof}

Let $v$ be the largest vertex in $N_H(\theta_{Q'})$  adjacent to a non-root vertex $u$ of a leaf block $B$ of $Q'$. Then $N_{V_2-X_v}(u) \subseteq \{v_1,v_2, v\}$.  For $i=1,2,\ldots, k-1$, $V(G_i) \subseteq X_v$. 

Note that $B$ may be contained in a block $B'$ of $Q$ which is a copy of $K_4$. In this case, $B'$ contains a vertex $w$ not visible to $G_k$, and hence $w \in X_v$ (by (R1)). So $B$ is either a cycle or a complete graph of order at most 2. 
   Thus $d_{Q-X_v}(u) \le 2$.
Hence $d_{G-X_v}(u) \le 5$, and $v$ is a protector of $Q$, a contradiction.

This completes the proof of Theorem \ref{thm-planarDP}, except that Lemma \ref{lem-nice} is not proved yet.

\section{Proof of Lemma \ref{lem-nice}}

For the purpose of using induction, we prove a slightly stronger result.

For a plane graph $G$,  we denote by $\theta^*_G$ the infinite face of $G$.   

\begin{definition}
    \label{def-verynice}
   Let $v^*$ be a vertex on the boundary of the infinite face $\theta^*_G$. A nice subgraph $H$ of $\Theta(G)$ is called   {\em very nice} with respect to $v^*$  if 
    $d_H(\theta^*_G) = d_{\Theta(G)}(\theta^*_G)$ and $d_H(v^*) =1$.
\end{definition}

\begin{lem}
    \label{lem-verynice}
    For every  plane graph $G$, for any vertex $v^*$ on the boundary of infinite face of $G$, $\Theta(G)$ has a very nice subgraph with respect to $v^*$.
\end{lem}
\begin{proof}
   We shall prove Lemma \ref{lem-verynice} by induction on the number of vertices of $G$.

   \medskip
   \noindent
   {\bf Case 1} $G$ is 2-connected.
   \medskip
   
   If $G=K_1$ or $K_2$, then $H=\Theta(G)$ is very nice with respect to $v^*$. 
   
   If $G \ne K_1, K_2$ is an outerplanar graph, then $G$ contains a cycle $C=(v_1,v_2,\ldots, v_k)$, bounding a finite face $\theta$,  such that $d_G(v_i)=2$ for $i=2,3,\ldots, k-1$, and $v^* \notin \{v_2, v_3, \ldots, v_{k-1}\}$. Let $G'=G-\{v_2, v_3, \ldots, v_{k-1}\}$. Then $\Theta(G)$ is obtained from $\Theta(G')$ by adding vertices $v_2,v_3, \ldots, v_{k-1}$ and $\theta$, and edges $v_i \theta$ for $i=1, 2, \ldots, k$ and $v_i \theta^*_G$ for $i=2,3,\ldots, k-1$. 
   
   By the induction hypothesis, $\Theta(G')$ has a very nice subgraph $H'$  with respect to $v^*$. Let $H$ be obtained from $H'$ by adding vertices $v_2,v_3, \ldots, v_{k-1}$ and $\theta$, and edges $v_i \theta, v_i \theta^*_G$   for $i=2,3,\ldots, k-1$.  Then $H$ is a very nice subgraph of $\Theta(G)$ with respect to $v^*$.

   Assume $G$ is not an outerplanar graph. 
   If $G$ has a degree 2 vertex $v \ne v^*$ not contained in a triangle, then let $G'$ be obtained from $G$ by removing the vertex $v$ and adding an edge connecting the two neighbours of $v$. Then $F(G)=F(G')$ and $\Theta(G)$ is obtained from $\Theta(G') $ by adding the vertex $v$ and the edges $v \theta_1, v \theta_2$, where $\theta_1,\theta_2$ are the two faces of $G$ incident to $v$. By induction hypothesis, $\Theta(G')$ has a very nice subgraph $H'$ with respect to $v^*$. Let $H$ be obtained from $H'$ by adding vertex $v$ and edges $v \theta_1, v \theta_2$. Then $H$ is a very nice subgraph of $\Theta(G)$ with respect to $v^*$.
 
Otherwise let $u$ be a vertex not on the boundary cycle of $G$, and let $G'=G-u$. Let $\theta_u$ be the face of $G'$ containing the vertex $u$. 
Let $C$ be the boundary cycle of $\theta_u$, and let 
 $u_1,u_2,\ldots, u_k$ be the neighbours of $u$, occurring in $C$ in this cyclic order in clockwise direction.  For $i=1,2,\ldots, k$, let $P_i$ be the subpath of $C$ from $u_i$ to $u_{i+1}$ (where let $u_{k+1}=u_1$) along the clockwise direction. So $u_i,u_{i+1}$ are the end vertices of $P_i$, and $\cup_{i=1}^k(P_i-\{u_i\})$ is a partition of $V(C)$.

    In $G$, $\theta_u$ is divided into $k$ faces $\theta_1, \theta_2, \ldots, \theta_k$, where $V(\theta_i) = V(P_i) \cup \{u\}$. 
    
    Let $H'$ be a very nice subgraph of $\Theta(G')$ with respect to $v^*$. Let $Z=N_{\Theta(G')}(\theta_u) -N_{H'}(\theta_u)$. Then $|Z| \le 2$. 
    
    Assume $Z \subseteq \{z_1,z_2\}$. Assume   $z_1 \in P_i - \{u_i\}$, and $z_2 \in P_j-\{u_j\}$. 
    If $i \ne j$, then let $H$ be obtained from $H'$ by deleting  $\theta_u$ (and hence all the  edges $\{v\theta_u: v \in V(\theta_u)\}$), and adding vertices $u, \theta_1,\theta_2,
    \ldots, \theta_k$ and adding  edges 
    $$\cup_{t=1}^k\{v\theta_t: v \in V(P_t)-\{u_t\}\} \cup \{u\theta_i,u\theta_j\} - \{z_1\theta_i,z_2\theta_j\}.$$ 
    If $i=j$, then without loss of generality, assume that 
$i=j=1$. Let 
$H$ be obtained from $H'$ by deleting  $\theta_u$ (and hence all the  edges $\{v\theta_u: v \in V(\theta_u)\}$), and adding vertices $u, \theta_1,\theta_2,
    \ldots, \theta_k$ and adding  edges   
    $$\cup_{t=1}^k\{v\theta_t: v \in V(P_t)-\{u_t\}\} \cup \{u\theta_1, u\theta_k, u_1\theta_1\} - \{z_1\theta_1,z_2\theta_1, u_1\theta_k\}.$$

    It is straightforward to verify that $H$ is a very nice subgraph of $\Theta(G)$ with respect to $v^*$.

\medskip
\noindent
{\bf Case 2} 
 $G$ is not 2-connected.
 \medskip

Note that $G$ may be connected or not connected. Let $B$ be a leaf block of $G$ (which could be a 2-connected component of $G$) such that
all vertices of $G-U(B)$ are contained in the infinite face of $B$ and $v^* \notin V(B)$. It is obvious that such a leaf block $B$ exists. Let $G'=G-U(B)$. By the induction hypothesis, $\Theta(G')$ has a very nice subgraph $H'$ with respect to $v^*$.  

Let $\theta_B$ be the face of $G'$ that contains $U(B)$. 

If $B$ is a 2-connected component of $G$, then $\Theta(G)$ is obtained from the disjoint union of $\Theta(G')$ and $\Theta(B)$ by identifying $\theta_B$ with $\theta^*_B$. Let $H''$ be a very nice subgraph of $\Theta(B)$ with respect to $v'$ for an arbitrary boundary vertex $v'$ of $B$. Let $H$ be obtained from the disjoint union of $H'$ and $H''$ by identifying $\theta_B$ with $\theta^*_B$. Then $H$ is a very nice subgraph of $\Theta(G)$ with respect to $v^*$. 

Otherwise $B$ contains a cut-vertex $v'$ of $G$, and $\Theta(G)$ is obtained from the disjoint union of $\Theta(G')$ and $\Theta(B)$ by identifying $\theta_B$ with $\theta^*_B$, and identifying the copy of $v'$ in $G'$ and the copy of $v'$ in $B$. Let $H''$ be a very nice subgraph of $\Theta(B)$ with respect to $v'$. Let $H$ be obtained from the disjoint union of $H'$ and $H''$ by identifying $\theta_B$ with $\theta^*_B$, and identifying the copy of $v'$ in $H'$ and the copy of $v'$ in $H''$.   Then $H$ is a very nice subgraph of $\Theta(G)$ with respect to $v^*$. 

This completes the proof of Lemma \ref{lem-verynice}.
\end{proof}

\section{Proof of Theorem \ref{thm-minorDP}}

Note that for any proper minor closed family $\mathcal{G}$ of graphs, there exist positive integers $s,t$ such that the complete bipartite graph $K_{s,t} \notin \mathcal{G}$. This is so because there is a complete graph $K_s \notin \mathcal{G}$ (for otherwise $K_s \in \mathcal{G}$ for all $s$, and hence every graph is in $\mathcal{G}$ and $\mathcal{G}$ is not proper), and $K_s$ is a minor of $K_{s,s}$. So, the integer $s$ as described in the theorem is well defined. 

Let $t$ be the minimum integer such that $K_{s,t} \notin \mathcal{G}$. 
Let $$q = 4^{s+1}s! st (s+t-1)+1, \text{ and }  k=2^{s+2}tq.$$ We shall prove that for every graph $G \in \mathcal{G}$, if $G$ is not a GDP tree, then $G$ is degree-truncated DP-$k$-colourable.

If $s=1$, then the graphs in $\mathcal{G}$ have maximum degree at most $t-1$.  For $G \in \mathcal{G}$ which is not a GDP-tree, $G$ is  degree-DP-colourable, and hence degree-truncated DP-$k$-colourable.

In the following, assume $s \ge 2$.

The lemmas \ref{lem-1b} and \ref{lem-2} below proved in \cite{Thomason} are needed in our proof.

\begin{lem}
	\label{lem-1b} Every non-empty graph $G$ with at least $2^{s+1}t|V(G)|$ edges has a 
	 $K_{s,t}$-minor.
\end{lem}

\begin{lem}
	\label{lem-2}
If $G$ is a bipartite graph with partite sets $A$ and $B$, and with at least $(s-1)|A|+4^{s+1}s!t|B|$ edges, then $G$ has a $K_{s,t}$-minor.
\end{lem}

The following is a consequence of Lemma \ref{lem-2}.

\begin{cor}
	\label{cor-1}
	Assume $G$ is a bipartite graph with partite sets $A$ and $B$, and each vertex in $A$ has degree at least $s$. If $G$ has no $K_{s,t}$-minor, then $|E(G)| < 4^{s+1}s! st|B|$ and consequently, $B$ has a vertex $v$ of degree at most $4^{s+1}s! st$.
\end{cor}
\begin{proof}
	Assume $G$ has no $K_{s,t}$-minor. By Lemma \ref{lem-2}, $|E(G)| < (s-1)|A|+4^{s+1}s!t|B|$. 
	So $|E(G)| - (s-1)|A| < 4^{s+1}s!t|B|$. 
	
	As each vertex in $A$ has degree at least $s$, we have $|E(G)| \ge s|A|$, which implies that $|E(G)| - (s-1)|A| \ge |E(G)|- \frac{s-1}{s}|E(G)| = |E(G)|/s$. Therefore 
	$|E(G)|  \le s (|E(G)| - (s-1)|A|) < 4^{s+1}s! st|B|$.
\end{proof}

Let $V_1 = \{v \in V(G): d(v) < k\}$ and $V_2 =V(G)-V_1$. 
Let $p=|V_2|$.  Assume $(L,M)$ is a cover of $G$ such that $|L(v)| = d_G(v)$ for $v \in V_1$, and $|L(v)|=k$ for $v \in V_2$.
We shall show that $G$ is $(L,M)$-colourable.

		By Lemma \ref{lem-1b}, each subgraph $H$ of $G$ has minimum degree at most $2^{s+2}t-1$, that is, $G$ is $(2^{s+2}t-1)$-degenerate. 
  Hence, the vertices of $V_2$ can be ordered as 
  $w_1,w_2, \ldots, w_p$  so that each $w_i$ has at most $2^{s+2}t-1$ neighbours $w_j$ with $j < i$.

\begin{lem}
    \label{lem-sub}
    There exists $L'(w_i) \subseteq L(w_i)$ such that 
    \begin{itemize}
        \item $|L'(w_i)|=q$  for $i=1,2,\ldots, p$, and 
        \item $L'(w_i) \cap N_{(L,M)}(L'(w_j)) = \emptyset$ for each edge $w_iw_j$ of $G[V_2]$.
    \end{itemize}
\end{lem}
\begin{proof}
    We choose subsets $L'(w_j) \subseteq L(w_j)$ for $j=1,2,\ldots, p$ recursively as follows:
    
     Let $L'(w_1)$ be any $q$-subset of $L(w_1)$. 

    Assume $i \ge 1$, and we have already chosen subsets $L'(w_j)$ for $j=1,2,\ldots, i$.

    Let $\{w_{j_1}, w_{j_2}, \ldots, w_{j_l}\}$ be the set of neighbours  of $w_{i+1}$ with $j_r < i+1$ for $r=1,2,\ldots, l$. Let 
    $L'(w_{i+1}) $ be any $q$-subset of $L(w_{i+1}) - N_{(L,M)} (\bigcup_{r=1}^l N_{(L,M)} (L' (w_{j_r}))$.

Since $l \le   2^{s+2}t-1$,  $|L(w_{i+1})|=2^{s+2}tq$, 
and $|\bigcup_{r=1}^l N_{(L,M)} (L' (w_{j_r}))| \le |\bigcup_{r=1}^l |L' (w_{j_r})| \le lq \le (2^{s+2}t-1)q$, the set $L'(w_{i+1}) $ is well defined. From the definition it follows that $|L'(w_i)|=q$ for $i=1,2,\ldots, p$ and $L'(w_i) \cap N_{(L,M)}(L'(w_j)) = \emptyset$ for any edge $w_iw_j$ of $G[V_2]$.
 \end{proof}

Let $L'(v)$ for $v \in V_2$ be chosen as in Lemma \ref{lem-sub} and $L'(v)=L(v)$ for $v \in V_1$. Let $M'$ be the restriction of
$M$ to $\bigcup_{v \in V(G)}L'(v)$.

We shall prove that $G$ is $(L',M')$-colourable.
Since for any edge $w_iw_j$ of $G[V_2]$, $M'_{w_iw_j} = \emptyset$, it is equivalent to proving that $G'=G-E(G[V_2])$ is $(L',M')$-colourable. In other words, we can treat $V_2$ as an independent set in the colouring process. 

Let $G_p$ be the bipartite graph obtained from $G'$ by contracting  each connected component $Q$ of $G'[V_1]$ into a vertex $v_Q$.  
Let $A_p$ be the partite set of $G_p$ consisting of vertices $v_Q$ for the connected components $Q$ of $G'[V_1]$, and let $B_p=V_2$.

Since $G$ is $s$-connected, each vertex $v \in A_p$ has degree $d_{G_p}(v) \ge s$. 

As $G_p$ has no $K_{s,t}$-minor, by Corollary \ref{cor-1},  there is a vertex $u_p \in B_p$ with $d_{G_p}(u_p) \le 4^{s+1}s!st$. Let $R_p = N_{G_p}(u_p)$.

Let $G_{p-1} = G_p- (R_p \cup \{u_p\})$. Again, each vertex in $A_{p-1}=A_p-R_p$ has $d_{G_{p-1}}(v) \ge s$, and $G_{p-1}$ is $K_{s,t}$-minor free. Therefore, $B_{p-1}=B_p-\{u_p\} $ has a vertex $u_{p-1}$ with $d_{G_{p-1}}(u_{p-1}) \le 4^{s+1}s!st$. Let $R_{p-1} = N_{G_{p-1}}(u_{p-1})$. 

Repeat this process, we obtain an ordering $u_1,u_2, \ldots, u_p$ of vertices of $V_2$, and a sequence   $R_1,R_2,\ldots, R_p$ of subsets of $A_p$ such that  
\begin{itemize}
	\item for each $i=1, 2,3,\ldots, p$,   $|R_i| \le 4^{s+1}s!st$,
 \item $A_p$ is the disjoint union of $R_1,R_2,\ldots, R_p$,
	\item for any $v_Q \in R_i$, $u_i$ is the last vertex in $V_2$ that has a neighbour in $Q$.
	\end{itemize}

Now we colour the vertices of $G$ one by one. Assume we have constructed a partial $(L',M')$-colouring $\phi$ of $G[X]$. The rules for choosing the next vertex to be coloured are  (R1) and (R2) described in the proof of Theorem \ref{thm-planarDP}.  (The notion $X_v, Y_v, \phi_v, \psi_v$ and the concept of safe components of $G[V_1]$ are defined the same way as in the proof of Theorem \ref{thm-planarDP}).

When (R1) is applied, the colour for $v$ is an arbitrary colour in $L^{\phi}(v)$. When (R2) is applied, the colour for $v$ is chosen as follows: 

Assume the next vertex to be coloured is $v=u_i \in V_2$. 

For each vertex $v_Q \in R_i$, if $Q$ is not safe with respect to $\phi_{u_i}$, then $Q-X_{u_i}$ is a GDP-tree. 
 Each block of $Q-X_{u_i}$ is a complete graph or a cycle. By (R1),  each vertex $v$ of $Q-X_{u_i}$ is adjacent to $u_i$ or is a cut-vertex of $Q-X_{u_i}$. Thus, $Q-X_{u_i}$ has a leaf block $B$ which has a non-root vertex $v$ which is adjacent to $u_i$. 
  If $B$ is a cycle, then $d_{G-X_{u_i}}(v) = d_{Q-X_{u_i}}(v)+1 = 3$. If $B$ is a complete graph, then 
  since $G$ is $K_{s,t}$-minor free, $|V(B)| \le s+t-1$, and hence  $d_{G-X_{u_i}}(v) = d_{Q-X_{u_i}}(v)+1 \le s+t-1$. 

  Since $$|L'(u_i)|=q>(s+t-1)\times 4^{s+1}s!st\ge (s+t-1)|R_i|,$$
  we can find a colour   $c \in L'(u_i)$
  so that by colouring $u_i$ with colour $c$, for each vertex $v_Q \in R_i$, $Q$ is safe with respect to $\psi_{u_i}$.
 
After all vertices of $V_2$ are coloured, all connected components of $G[V_1]$ are safe. Hence the partial $(L',M')$-colouring can be extended to an $(L',M')$-colouring of $G$.

  For any surface $\Sigma$, it is well known and easy to prove that there is a constant $t$ such that $K_{3,t}$ is not embeddable in $\Sigma$. Hence, there is a constant $k$ such that every 3-connected non-complete graph embedded in $\Sigma$ is degree-truncated DP-$k$-colourable.

This completes the proof of Theorem \ref{thm-minorDP}.


	\end{document}